\documentclass[11pt]{article}
\usepackage[a4paper, tmargin=3cm, bmargin=3.0cm, lmargin=2.0cm, rmargin=2.0cm, textheight=24cm, textwidth=16cm]{geometry}
\usepackage{setspace}
\usepackage{amsfonts}
\usepackage{epsfig}
\usepackage{latexsym}
\usepackage{amsthm}
\usepackage{amsmath}
\usepackage{amssymb}
\usepackage{array}
\usepackage{enumerate}
\usepackage{graphicx}
\usepackage{breqn}

\newtheorem{theo}{Theorem}[section]

\newtheorem{exam}{Example}[section]
\newtheorem{definition}{Definition}[section]

\newtheorem{coro}{Corollary}[section]

\newtheorem{re}{Remark}[section]

\title{On distance and Laplacian matrices of trees with matrix weights}

\author{Fouzul Atik\thanks{Theoretical Statistics and Mathematics Unit, Indian Statistical Institute, New Delhi 110 016, India. Email: fouzulatik@gmail.com}  \and M. Rajesh Kannan\thanks{Department of Mathematics, Indian Institute of Technology Kharagpur, Kharagpur, India. Email: rajeshkannan@maths.iitkgp.ernet.in, rajeshkannan1.m@gmail.com } \and R. B. Bapat\thanks{Theoretical Statistics and Mathematics Unit, Indian Statistical Institute, New Delhi 110 016, India. Email: rbb@isid.ac.in}
}

\date{\today}
\begin{document}
\maketitle
\baselineskip=0.25in

\begin{abstract}
The \emph{distance matrix} of a simple connected graph $G$ is
$D(G)=(d_{ij})$, where $d_{ij}$ is the distance between the vertices
$i$ and $j$ in $G$. We consider a weighted tree $T$ on $n$ vertices
with edge weights are square matrix of same size. The distance
$d_{ij}$ between the vertices $i$ and $j$ is the sum of the weight
matrices of the edges in the unique path from $i$ to $j$. In this
article we establish a characterization for the trees in terms of
rank of (matrix) weighted Laplacian matrix associated with it. Then
we establish a necessary and sufficient condition for the distance
matrix $D$, with matrix weights, to be invertible and the formula
for the inverse of $D$, if it exists. Also we study some of the
properties of the distance matrices of matrix weighted trees in
connection with the Laplacian matrices, g-inverses and eigenvalues.
\end{abstract}



\section{Introduction}\label{section1}
Consider an $n$-vertex connected graph $G=(V,E)$, where
$V=\{1,2,\dots,n\}$ is the vertex set and
$E=\{e_1,e_2,\dots,e_m\}$ is the edge set. The \emph{distance
matrix} $D(G)$ of $G$ is an $n\times n$ matrix $(d_{ij})$ such that
$d_{ij}$ is the distance (length of a shortest path) between the
  vertices $i$ and $j$ in $G$. Let $T$ be a tree with $n$
vertices and let $W_i$ be the $s\times s$ edge weight matrix
associated with the edge $e_i,~i=1,2,\dots,n-1$. If $i$ and $j$ are
vertices of $T$, then there is a unique path from $i$ to $j$ in $T$.
The distance $d_{ij}$ between the vertices $i$ and $j$ is defined to
be the sum of the weights matrices of the edges in the unique path
connecting them. Then the distance matrix $D$ of $T$ is a block
matrix of order $ns\times ns$. Usually the Laplacian matrix of the weighted
tree $T$ with matrix weight is the block matrix $(l_{ij})$, where
$l_{ij}$ is equal to the sum of the weight matrices of the edges
incident on vertex $i$ if $i=j$, equal to $-W$ if there is an edge
between $i$ and $j$ and weight of that edge is $W$, and zero matrix
otherwise. In this article, we consider the Laplacian matrix
obtained by replacing each edge weight matrix of the tree by its
inverse.

The distance matrix of graphs, in particular trees, have been
studied by researchers for many years. An early remarkable result of
Graham and Pollak \cite{grah2} states that, if $T$ is a tree on $n$
vertices, then the determinant of the distance matrix $D$ of $T$ is
$(-1)^{n-1}(n-1)2^{n-2}$, which is independent of the structure of
the underlying tree. Subsequently, Graham and Lov\'{a}sz
\cite{grah1} derived a formula for the inverse of the distance
matrix of a tree. An extension of these two results for the weighted
trees, where the weights are being positive scalars, were obtained
by Bapat \emph{et al.} \cite{bapat1}. In \cite{bapat2}, Bapat
obtained the determinant of the distance matrix of a weighted tree
where weights are arbitrary matrices of fixed order. In
\cite{balaji}, Balaji and Bapat determined the inverse of the
distance matrix of a weighted tree where weights are positive
definite matrices. Bapat \emph{et al.,} considered the $q$-distance
matrix of an unweighted tree and gave formulae for the inverse and
the determinant\cite{bapat3}. Similar type of results were
established for distance matrices of trees with weights are from
rings \cite{zhou, zhou1}.

In this paper we consider weighted trees with the edge weights are
 matrices of same size. We establish a characterization for trees in terms of rank of (matrix) weighted Laplacian matrix  associated with it(Theorem \ref{char-tree}). Then  we derive a necessary and sufficient condition
for the distance matrix to be invertible and the formula for the inverse of the distance matrix, if it exist(Theorem \ref{main}). Also we study some of the properties of the distance matrices of matrix weighted trees in connection with the Laplacian matrices(Theorem \ref{extra}) and g-inverses(Theorem \ref{ginv}). Finally, we derive an interlacing inequality for the eigenvalues of distance and Laplacian matrices for the case of positive definite matrix weights(Theorem \ref{interlace}).

This article is organized as follows: In section \ref{tree-char}, we
study some of the properties of the matrix weighted Laplacian
matrices of graphs and provide a characterization for trees in terms
rank of weighted Laplacian matrix associated with it. In section
\ref{inverse_formula}, we establish a formula for the inverse of
distance matrix of  weighted trees whose weights are invertible
matrices. Using this, we prove some properties of distance matrices
related to Laplacian matrices, g-inverse and eigenvalues.

\section{A characterization for trees}\label{tree-char}

In this section, first  we define the notion of incidence matrix of a
weighted graph whose edges are assigned positive definite matrix
weights.

\begin{definition} Let $G$ be a connected weighted graph with $n$ vertices and $m$ edges and all
the edge weights of $G$ are positive definite matrices of order
$s\times s$. We assign an orientation to each edge of $G$. Then the
vertex edge incidence matrix $Q$ of $W$ is the $ns\times ms$ block
matrix whose row and column blocks are index by the vertices and
edges of $G$, respectively. The $(i,j)$th block of $Q$ is zero
matrix if the $i$th vertex and the $j$th edge are not incident and
it is $\sqrt{W_j}^{-1}$ (respectively, $-\sqrt{W_j}^{-1}$) if the $i$th vertex
and the $j$th edge are incident, and the edge originates
(respectively, terminates) at the $i$th vertex, where $W_j$ denotes
the weight of the $j$th edge.
\end{definition}

In this case, it is easy to see that, the Laplacian matrix $L$ can be written
as $L=QQ^T$.

Recall that, the Kronecker product of matrices $A=(a_{ij})$ of size
$m\times n$ and $B$ of size $p\times q$, denoted by $A\otimes B$, is
defined to be the $mp\times nq$ block matrix $(a_{ij}B)$. It is
known that for matrices $M$, $N$, $P$ and $Q$ of suitable sizes,
$MN\otimes PQ=(M\otimes P)(N\otimes Q)$\cite{hor}.

\begin{theo}
If $G$ is a connected weighted graph on $n$ vertices such that the
edge weights are $s\times s$ positive definite matrices, then rank
of $Q$ as well as rank of $L$ is $(n-1)s$.
\end{theo}
\begin{proof}
We have $Q^T$ is a matrix of order $ms\times ns$, where $m$ is the
number of edges in $G$. Let $X$ is a vector in the null space in
$Q^T$ and let $X=(X_1,X_2,\cdots,X_n)^T$ be a partition where each
$X_i$ is of order $1\times s$. Then $Q^TX=0$ gives
$(\sqrt{W_i})^{-1}(X_i-X_j)=0$, where $W_i$ is the weight of the
$(i,j)$th edge in $G$. Since $W_i$ is positive definite, $X_i=X_j$.
As the graph $G$ is connected, we must have $X_i=X_j$ for all
$i,j=1,2,\cdots,n$. Thus $X$ must be of the form $e^T\otimes X_1^T$,
where $e$ is a all one vector of order $n$. So the dimension of null
space of $Q^T$ is at most $s$ and thus rank of $Q^T$ is at least
$(n-1)s$. Again the block row sum of $Q^T$ is zero, gives rank of
$Q^T$ is at most $(n-1)s$. Hence rank of $Q^T$ and $Q$ is $(n-1)s$.
Also as the edge weights of the graphs are positive definite
matrices, so $L=QQ^T$ and rank of $L$ is $(n-1)s$.
\end{proof}
Next we give an example of a weighted graph (not a tree) with
weights are nonsingular matrix for which rank of $L$ is not
$(n-1)s$.
\begin{exam}\label{eaxm2}\emph{Consider the cycle of order 4 (Figure 1) , where the edge weights are}
$W_1=\left[
       \begin{array}{cc}
         1 & 0 \\
         0 & 1 \\
       \end{array}
     \right],
 W_2=\left[
       \begin{array}{cc}
         0 & 1 \\
         1 & 0 \\
       \end{array}
     \right],
 W_3=\left[
       \begin{array}{cc}
         1 & 0 \\
         0 & 1 \\
       \end{array}
     \right], \emph{and}~
 W_4=\left[
       \begin{array}{cc}
         0 & 1 \\
         1 & 0 \\
       \end{array}
     \right]
$ \emph{and then the Laplacian matrix is given by}
\begin{eqnarray*}L&=&\left[
                                       \begin{array}{cccc}
                                         W_1^{-1}+W_4^{-1} & -W_1^{-1} & 0 & -W_4^{-1} \\
                                         -W_1^{-1} & W_1^{-1}+W_2^{-1} & -W_2^{-1} & 0 \\
                                         0 & -W_2^{-1} & W_2^{-1}+W_3^{-1} & -W_3^{-1} \\
                                         -W_4^{-1} & 0 & -W_3^{-1} & W_3^{-1}+W_4^{-1} \\
                                       \end{array}
                                     \right]\\
            &=&\left[
                                               \begin{array}{cccccccc}
                                                 1 & 1 & -1 & 0 & 0 & 0 & 0 & -1 \\
                                                 1 & 1 & 0 & -1 & 0 & 0 & -1 & 0 \\
                                                 -1 & 0 & 1 & 1 & 0 & -1 & 0 & 0 \\
                                                 0 & -1 & 1 & 1 & -1 & 0 & 0 & 0 \\
                                                 0 & 0 & 0 & -1 & 1 & 1 & -1 & 0 \\
                                                 0 & 0 & -1 & 0 & 1 & 1 & 0 & -1 \\
                                                 0 & -1 & 0 & 0 & -1 & 0 & 1 & 1 \\
                                                 -1 & 0 & 0 & 0 & 0 & -1 & 1 & 1 \\
                                               \end{array}
                                             \right]
\end{eqnarray*} \emph{Then one can verify that rank of $L$ is 5($\neq
(n-1)s$).}
\end{exam}
\newpage
\begin{figure}
\centering
\includegraphics[width=1.5in]{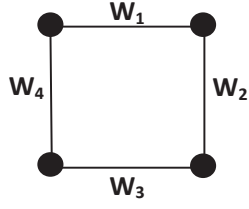}
\caption{Cycle of order 4}
\end{figure}
From the previous example we can see that for a weighted graph (not
a tree) with weights are nonsingular matrices for which rank of $L$
may not be $(n-1)s$. However from the theorem below, it is
guaranteed that if the graph is tree then for nonsingular matrix
weights of order $s$, then the rank of $L$ is $(n-1)s$.
\begin{theo}\label{weig-inv-tree}
If $T$ is a weighted tree on $n$ vertices such that the edge weights
are $s\times s$ nonsingular matrices, then rank of $L$ is $(n-1)s$.
\end{theo}
\begin{proof}
To prove the result we use induction on the number of vertices $n$
of the tree $T$. For $n=2$, $T$ is just an edge and let weight on
this edge is $W$, which is a nonsingular matrix of order $s$. Then
clearly rank of $L=\left[\begin{array}{cc}
W_1^{-1} & -W_1^{-1} \\
-W_1^{-1} & W_1^{-1} \\
\end{array}
\right]$ is $s$.

Assume that the result is true for $n=k-1$. Let $T$ be a tree on $k$
vertices. Consider a pendent vertex $v_1$ and the vertex $v_2$ which
is adjacent to $v_1$ and let weight of the edge $(v_1,v_2)$ is $W$.
Let $L$ be the Laplacian matrix of $T$ in which vertex ordering is
$v_1,~v_2$ and followed by the other vertices of $T$. Let $T^*$ be
the tree obtained by removing the vertex $v_1$ from $T$ and $L^*$ be
the corresponding Laplacian matrix of $T^*$. Then
\begin{eqnarray*}L=\left[\begin{array}{cc|ccc}
                                         W^{-1} & -W^{-1} & 0 & \cdots & 0 \\
                                        -W^{-1} & W^{-1}+A & * & \cdots & * \\
                                        0 & * & * & \cdots & * \\
                                        \cdots & \cdots & \cdots & \cdots & \cdots \\
                                        0 & * & * & \cdots & *
                                       \end{array}
                                     \right]=\left[\begin{array}{c|c}
W^{-1} & -W^{-1}~0~0~\cdots~0 \\\hline
-W^{-1} &  \\
0 & e_1e_1^T\otimes W^{-1}+L^* \\
\cdots &  \\
0 & \\
\end{array}
\right],\end{eqnarray*} where $A$ is the first diagonal block of the
matrix $L^*$. Now by induction hypothesis rank of $L^*$ is $(n-2)s$.
As the block row sum of $L^*$ is zero, so we have the last $(n-2)s$
columns of $L^*$ are linearly independent. Then last $(n-2)s$
columns of $L$ are also linearly independent. Again as $W$ is
nonsingular, so all the $s$ linearly independent columns of $L$
corresponding to the vertex $v_2$ is linearly independent with the
last $(n-2)s$ columns of $L$. Hence we get $L$ has $(n-1)s$ number
of linearly independent columns so that the rank of $L$ is $(n-1)s$.
\end{proof}
In the following theorem, we prove that for a graph which is
not a tree, there exist an assignment of nonsingular matrix weights for its edges
such that rank of $L$ is less than $(n-1)s$.
\begin{theo}\label{weig-inv-not-tree}
Let $G$ be a connected weighted graph on $n$ vertices, which is not
a tree. Then there exist an assignment of $s\times s$ nonsingular
matrix weights on the the edges of $G$ such that rank of $L$ is less
than $(n-1)s$.
\end{theo}
\begin{proof}
To prove the result, it is sufficient to give an assignment of non
zero scaler weights on the edges of the graph $G$ such that rank of
$L$ is less than $(n-1)$. As $G$ is not a tree then $G$ contains at
least one cycle. Let $(i,j)$ be an edge on a cycle of $G$. Now
assign a nonzero scaler weight $w$ to the edge $(i,j)$ and for other
edges assign the weight $1$. By matrix tree theorem (weighted
version), any cofactor of $L$ is the sum of the weights of all
spanning trees of $G$. Note that in this case weight of a tree means
the product of all edge weights of the tree. Then it is clear that
any cofactor of $L$ is equal to the form $aw+b$, where $a$ is the
number of spanning trees which contain the edge $(i,j)$ and $b$ is
the number of spanning trees which does not contain the edge
$(i,j)$. As $(i,j)$ is not a cut edge, both $a,b$ are nonzero. Now
we consider $w=-\frac{b}{a}$, which is a non zero negative number.
Then each cofactor of $L$ is becomes zero, so that each minor of
order $n-1$ becomes zero. Hence in this case rank of $L$ is less
than $(n-1)$.
\end{proof}
We illustrate the above theorem by the following example.
\begin{exam}\emph{Consider the weighted graph $G$ (not a tree) in Figure 2. Here in the edge $(1,3)$ we
assign weight $w$ and all other edges assign the weight 1. Then the
corresponding Laplacian matrix is given by
\begin{eqnarray*}L=\left[\begin{array}{cccc}
                                         w+2 & -1 & -w & -1 \\
                                        -1 & 2 & -1 & 0 \\
                                        -w & -1 & 2+w & -1 \\
                                        -1 & 0 & -1 & 2
                                       \end{array}
                                     \right].\end{eqnarray*}
Now $T_1,~T_2,~T_3$ and $T_4$ are the spanning trees of $G$ which
contain the edge $(1,3)$ and $T_5,~T_6,~T_7$ and $T_8$ are the
spanning trees of $G$ which does not contain the edge $(1,3)$. The
weight of each of the tree $T_1,~T_2,~T_3$ and $T_4$ is $w$ and
weight of each of the tree $T_5,~T_6,~T_7$ and $T_8$ is $1$. Sum of
the weights of all spanning trees of $G$ is $4w+4$. One can verify
that each cofactor of $L$ is also $4w+4$. Now if we take $w=-1$,
then rank of $L$ is $2~(< n-1)$.}
\end{exam}
\begin{figure}
\centering
\includegraphics[width=5in]{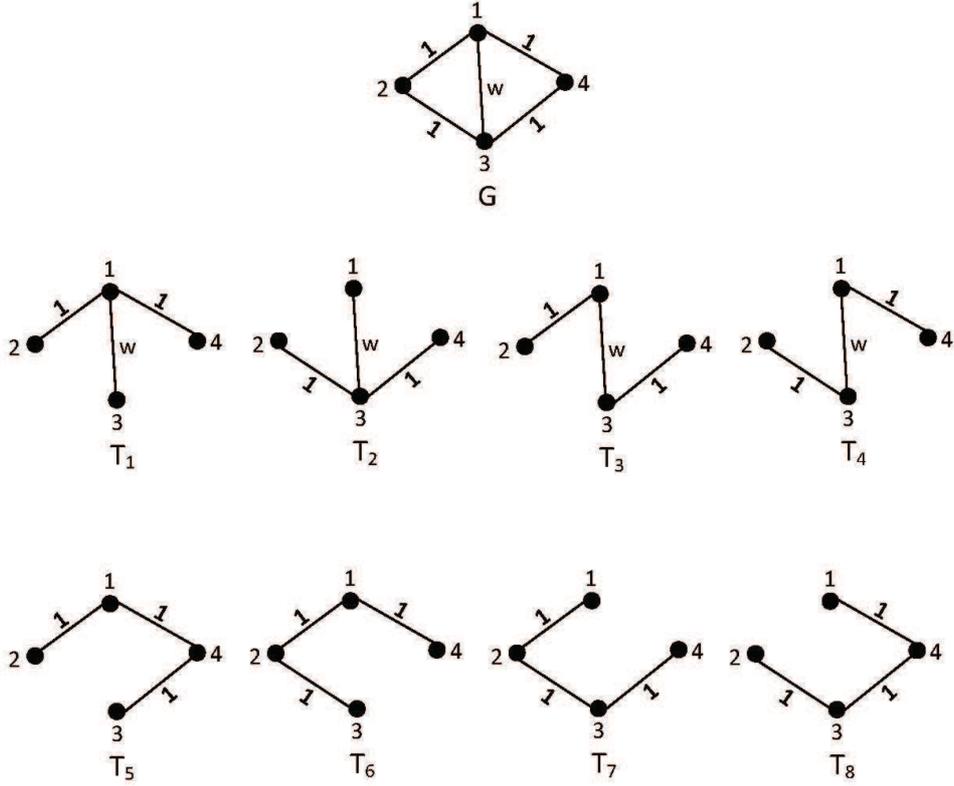}
\caption{Graph $G$ and its spanning trees}
\end{figure}
Now using Theorem \ref{weig-inv-tree} and Theorem
\ref{weig-inv-not-tree}, we can characterize the trees.
\begin{theo}\label{char-tree}

let $G$ be a graph on $n$ vertices. Then for all choices of
nonsingular weight matrices of the edges, rank of the associated
weighted Laplacian matrix $L$ is $(n-1)s$(where $s$ is the size of
the weight matrices) if and only if $G$ is a tree.
\end{theo}

\section{Inverse of distance matrices of trees with matrix weights}\label{inverse_formula}
In this section first we state the formula for the determinant of
the distance matrix of a tree with matrix weight which is obtained
in \cite{bapat2}.
\begin{theo}\label{det}
Let $T$ be a tree with $n$ vertices and let $W_i$ be the $s\times s$
edge weight matrix associated with the edge $e_i,~i=1,2,\dots,n-1$,
and let $D$ be the distance matrix of $T$. Then
$$det~ D=(-1)^{(n-1)s}2^{(n-2)s}det~
(\prod_{i=1}^{n-1}W_i)det ~(\sum_{i=1}^{n-1}W_i).$$
\end{theo}
\begin{theo}\label{main}
Let $T$ be a tree with $n$ vertices and  let $W_i$ be the $s\times
s$ edge weight matrices for  $i=1,2,\dots,n-1$, and let $D$ be the
distance matrix of $T$. Then $D$ is invertible if and only if the
weight matrices $W_i,~i=1,2,\dots,n-1$, and $(\sum_{i=1}^{n-1}W_i)$
are invertible. Let $D$ be invertible and let $L$ denote the
Laplacian matrix for weighting of $T$ in which each edge weight
matrix is replaced by its inverse. For $i=1,2,\dots,n$, set
$\delta_i=2-d_i$, where $d_i$ is the degree of the vertex $i$, and
$\delta^T=[\delta_1,\delta_2,\dots,\delta_n]$. Then the inverse of
$D$ is given by
$$D^{-1}=-\frac{1}{2}L+\frac{1}{2}\delta\delta^T\otimes(\sum_{i=1}^{n-1}W_i)^{-1}.$$
\end{theo}
\begin{proof}
From Theorem \ref{det}, it is easily verifiable that $D$ is
invertible if and only if $W_i,~i=1,2,\dots,n-1$, and
$(\sum_{i=1}^{n-1}W_i)$ are invertible.

For the second part, we use induction on the number of vertices $n$ of the tree $T$. For $n=2$, $T$ is just
an edge. Then

$D=\left[\begin{array}{cc}
0 & W_1 \\
W_1 & 0 \\
\end{array}\right]
,~L=\left[\begin{array}{cc}
W_1^{-1} & -W_1^{-1} \\
-W_1^{-1} & W_1^{-1} \\
\end{array}
\right]$ and $\delta=\left[
                      \begin{array}{c}
                        1 \\
                        1 \\
                      \end{array}
                    \right]
$. Then
\begin{eqnarray*}-\frac{1}{2}L+\frac{1}{2}\delta\delta^T\otimes
W_1^{-1}&=&-\frac{1}{2}\left[\begin{array}{cc}
W_1^{-1} & -W_1^{-1} \\
-W_1^{-1} & W_1^{-1} \\
\end{array}
\right]+\frac{1}{2}\left[\begin{array}{cc}
W_1^{-1} & W_1^{-1} \\
W_1^{-1} & W_1^{-1} \\
\end{array}
\right]\\
&=&\left[\begin{array}{cc}
0 & W_1^{-1} \\
W_1^{-1} & 0 \\
\end{array}
\right]=D^{-1}.\end{eqnarray*} So for $n=2$, the proof is complete.
For $n\geq3$, we assume that the inverse formula holds for trees of
order $n$. Let $T^*$ be a weighted tree on $n+1$ vertices, say
$1,2,\dots,n+1$, and the edge weights $W_i,~i=1,2,\dots,n$ are such
that  the weight matrices and $(\sum_{i=1}^{n}W_i)$ are invertible.
We consider a pendent vertex of $T^*$ and index it as $n+1$ and the
vertex adjacent to $n+1$ is indexed by $n$. Also let the weight of
the edge with end vertices $n$ and $n+1$ be $W_{n}$. We form a new
weighted tree $T$ by deleting the vertex $n+1$. We first assume that
$(\sum_{i=1}^{n-1}W_i)$ is invertible. Then by the induction
hypothesis, inverse formula for $D^{-1}$ is true for the tree $T$.
Let $D,~L$, and $\delta$ be the corresponding quantities for $T$ and
$D^*,~L^*$, and $\delta^*$ for $T^*$. Let $e_n$, $\textbf{1}$ and
$I$ denote the standard $n^{th}$ unit basis vector, the column
vector with all entries are equal to $1$ and the identity matrix of
appropriate size, respectively. Then, we have\vspace{0.25cm}
\begin{eqnarray*}L^*=\left[\begin{array}{c|c}
L+e_n e_n^T\otimes W_n^{-1} & -e_n\otimes W_n^{-1}\\\hline
-e_n^T\otimes W_n^{-1} & W_n^{-1}\\
\end{array}
\right],~ \delta^*=\left[
                     \begin{array}{c}
                       \delta-e_n \\\hline
                       1 \\
                     \end{array}
                   \right].
\end{eqnarray*}
and
\begin{eqnarray*}D^*=\left[\begin{array}{c|c}
D & D(e_n\otimes I)+\textbf{1}\otimes W_n \\\hline
(e_n^T\otimes I)D+\textbf{1}^T\otimes W_n & \textbf{0}\\
\end{array}
\right]
\end{eqnarray*}
We assume that $H_{n-1}=\sum_{i=1}^{n-1}W_i$ and
$H_{n}=\sum_{i=1}^{n}W_i$. Then
\begin{eqnarray}-\frac{1}{2}L^*+\frac{1}{2}\delta^*{\delta^*}^T\otimes(\sum_{i=1}^{n}W_i)^{-1}
\nonumber &=&-\frac{1}{2}L^*+\frac{1}{2}\delta^*{\delta^*}^T\otimes H_n^{-1}\\
\label{lhs}&=&\left[\begin{array}{c|c} P & Q
\\\hline
R & S \\
\end{array}
\right]
\end{eqnarray} where $P,~Q,~R,$ and $S$ are given by
\begin{eqnarray*}P&=&-\frac{1}{2}L-\frac{1}{2}e_n e_n^T\otimes
W_n^{-1}+\frac{1}{2}(\delta\delta^T-e_n\delta^T-\delta e_n^T+e_n
e_n^T)\otimes H_n^{-1}\\
Q&=&\frac{1}{2}e_n\otimes
W_n^{-1}+\frac{1}{2}(\delta-e_n)\otimes H_n^{-1}\\
R&=&\frac{1}{2}e_n^T\otimes
W_n^{-1}+\frac{1}{2}(\delta^T-e_n^T)\otimes
H_n^{-1}\\
S&=&-\frac{1}{2}W_n^{-1}+\frac{1}{2}H_n^{-1}
\end{eqnarray*}
Now one can verify that
\begin{eqnarray*}
\left[\begin{array}{c|c} I_n\otimes I_s & \textbf{0} \\\hline
-e_n^T\otimes I_s & I_s \\
\end{array}
\right]\left[\begin{array}{c|c} D & D(e_n\otimes
I)+\textbf{1}\otimes W_n \\\hline
(e_n^T\otimes I)D+\textbf{1}^T\otimes W_n & \textbf{0}\\
\end{array}
\right]\left[\begin{array}{c|c} I_n\otimes I_s & -e_n\otimes I_s
\\\hline
\textbf{0} & I_s \\
\end{array}
\right]
\\
=\left[\begin{array}{c|c} D & \textbf{1}\otimes W_n
\\\hline
\textbf{1}^T\otimes W_n & -2W_n\\
\end{array}
\right],
\end{eqnarray*}
and this implies
\begin{eqnarray}
\nonumber{D^*}^{-1}&=&\left[\begin{array}{c|c} D & D(e_n\otimes
I)+\textbf{1}\otimes W_n
\\\hline
(e_n^T\otimes I)D+\textbf{1}^T\otimes W_n & \textbf{0}\\
\end{array}
\right]^{-1}
\\
\label{inverse}&=&\left[\begin{array}{c|c} I_n\otimes I_s &
-e_n\otimes I_s
\\\hline
\textbf{0} & I_s \\
\end{array}
\right]\left[\begin{array}{c|c} D & \textbf{1}\otimes W_n
\\\hline
\textbf{1}^T\otimes W_n & -2W_n\\
\end{array}
\right]^{-1}\left[\begin{array}{c|c} I_n\otimes I_s & \textbf{0}
\\\hline
-e_n^T\otimes I_s & I_s \\
\end{array}
\right]
\end{eqnarray}
In the partitioned matrix $\left[\begin{array}{c|c} D &
\textbf{1}\otimes W_n
\\\hline
\textbf{1}^T\otimes W_n & -2W_n\\
\end{array}
\right]$, the Schur complement of the block $D$ is given by
\begin{eqnarray*}-2W_n-(\textbf{1}^T\otimes W_n)D^{-1}(\textbf{1}\otimes W_n)
&=&-2W_n-(\textbf{1}^T\otimes
W_n)\left(-\frac{1}{2}L+\frac{1}{2}\delta\delta^T\otimes
H_{n-1}^{-1}\right)(\textbf{1}\otimes W_n)
\\&&~~~~~~~~~~~~~~~~~~~~~~~~~~~~~~~~~~~~~~~~\mbox{[By induction
hypothesis]}\\
&=&-2W_n-(\textbf{1}^T\otimes
W_n)\left[-\frac{1}{2}L(\textbf{1}\otimes
W_n)+\frac{1}{2}(\delta\delta^T\otimes
H_{n-1}^{-1})(\textbf{1}\otimes
W_n)\right]\\
&=&-2W_n-(\textbf{1}^T\otimes
W_n)\left[\textbf{0}+\frac{1}{2}(\delta\delta^T\textbf{1})\otimes(H_{n-1}^{-1}W_n)\right]\\
&&~~~~~~~~~~~~~~~~~~~~~~~~~~~~~\mbox{[As each block row sum of $L$ is zero]}\\
&=&-2W_n-(\textbf{1}^T\otimes
W_n)\left[\delta\otimes(H_{n-1}^{-1}W_n)\right]\\
&&~~~~~~~~~~~~~~~~~~~~~~~~~~~~~~~~~~~~~~~~\mbox{[As $\delta^T\textbf{1}=\textbf{1}^T\delta=2$]}\\
&=&-2W_n-(\textbf{1}^T\delta)\otimes
(W_nH_{n-1}^{-1}W_n)\\
&=&-2W_n-2W_nH_{n-1}^{-1}W_n\end{eqnarray*} Now using a well known
formula for the inverse of a partition matrix we get
\begin{eqnarray}
\label{inverse1}\left[\begin{array}{c|c} D & \textbf{1}\otimes W_n
\\\hline
\textbf{1}^T\otimes W_n & -2W_n\\
\end{array}
\right]^{-1}=\left[\begin{array}{c|c} P_0 & Q_0
\\\hline
R_0 & S_0\\
\end{array}
\right]
\end{eqnarray}where $P_0,~Q_0,~R_0,$ and $S_0$ are given by
\begin{eqnarray*}P_0&=&D^{-1}+D^{-1}(\textbf{1}\otimes
W_n)(-2W_n-2W_nH_{n-1}^{-1}W_n)^{-1}(\textbf{1}^T\otimes W_n)D^{-1}\\
Q_0&=&-D^{-1}(\textbf{1}\otimes
W_n)(-2W_n-2W_nH_{n-1}^{-1}W_n)^{-1}\\
R_0&=&-(-2W_n-2W_nH_{n-1}^{-1}W_n)^{-1}(\textbf{1}^T\otimes W_n)D^{-1}\\
S_0&=&(-2W_n-2W_nH_{n-1}^{-1}W_n)^{-1}
\end{eqnarray*}
Now we simplify $P_0,~Q_0,~R_0,$ and $S_0$ one by one.
\begin{eqnarray*}P_0&=&D^{-1}+\delta\otimes(H_{n-1}^{-1}W_n)(-2W_n-2W_nH_{n-1}^{-1}W_n)^{-1}\delta^T\otimes(W_n H_{n-1}^{-1})\\
&=&D^{-1}+\delta\delta^T\otimes(H_{n-1}^{-1}W_n)(-2W_n-2W_nH_{n-1}^{-1}W_n)^{-1}(W_n H_{n-1}^{-1})\\
&=&D^{-1}-\frac{1}{2}\delta\delta^T\otimes(H_{n-1}^{-1}W_nH_n^{-1})\\
Q_0&=&-\delta\otimes(H_{n-1}^{-1}W_n)(-2W_n-2W_nH_{n-1}^{-1}W_n)^{-1}=\frac{1}{2}\delta\otimes
H_n^{-1}\mbox{ and } ~R_0=\frac{1}{2}\delta^T\otimes H_n^{-1}\\
S_0&=&(-2W_n-2W_nH_{n-1}^{-1}W_n)^{-1}=-\frac{1}{2}(W_n^{-1}-H_n^{-1})
\end{eqnarray*}
Thus using (\ref{inverse}) and (\ref{inverse1}) we get
{\small
\begin{eqnarray} \nonumber
&&{D^*}^{-1}=\left[\begin{array}{c|c} I_n\otimes I_s & -e_n\otimes
I_s
\\\hline
\textbf{0} & I_s \\
\end{array}
\right]\left[\begin{array}{c|c}
D^{-1}-\frac{1}{2}\delta\delta^T\otimes(H_{n-1}^{-1}W_nH_n^{-1}) &
\frac{1}{2}\delta\otimes H_n^{-1}
\\\hline
\frac{1}{2}\delta^T\otimes H_n^{-1} & -\frac{1}{2}(W_n^{-1}-H_n^{-1})\\
\end{array}
\right]\left[\begin{array}{c|c} I_n\otimes I_s & \textbf{0}
\\\hline
-e_n^T\otimes I_s & I_s \\
\end{array}
\right]\\
\nonumber&&=\left[\begin{array}{c|c}
D^{-1}-\frac{1}{2}\delta\delta^T\otimes(H_{n-1}^{-1}W_nH_n^{-1})-\frac{1}{2}e_n\delta^T\otimes
H_n^{-1} & \frac{1}{2}\delta\otimes
H_n^{-1}+\frac{1}{2}e_n\otimes(W_n^{-1}-H_n^{-1})
\\\hline
\frac{1}{2}\delta^T\otimes H_n^{-1} & -\frac{1}{2}(W_n^{-1}-H_n^{-1})\\
\end{array}
\right]\left[\begin{array}{c|c} I_n\otimes I_s & \textbf{0}
\\\hline
-e_n^T\otimes I_s & I_s \\
\end{array}
\right]\\
\nonumber&&=\left[\begin{array}{c|c}
-\frac{1}{2}L+\frac{1}{2}\delta\delta^T\otimes
(H_{n-1}^{-1}-H_{n-1}^{-1}W_{n}^{-1}H_{n}^{-1})-\frac{1}{2}e_n\delta^T\otimes
H_n^{-1} & \frac{1}{2}\delta\otimes
H_n^{-1}+\frac{1}{2}e_n\otimes(W_n^{-1}-H_n^{-1})
\\\hline
\frac{1}{2}\delta^T\otimes H_n^{-1} & -\frac{1}{2}(W_n^{-1}-H_n^{-1})\\
\end{array}
\right]\left[\begin{array}{c|c} I_n\otimes I_s & \textbf{0}
\\\hline
-e_n^T\otimes I_s & I_s \\
\end{array}
\right]\\
\label{rhs}&&=\left[\begin{array}{c|c}
-\frac{1}{2}L+\frac{1}{2}\delta\delta^T\otimes
H_n^{-1}-\frac{1}{2}e_n\delta^T\otimes H_n^{-1} &
\frac{1}{2}\delta\otimes
H_n^{-1}+\frac{1}{2}e_n\otimes(W_n^{-1}-H_n^{-1})
\\\hline
\frac{1}{2}\delta^T\otimes H_n^{-1} & -\frac{1}{2}(W_n^{-1}-H_n^{-1})\\
\end{array}
\right]\left[\begin{array}{c|c} I_n\otimes I_s & \textbf{0}
\\\hline
-e_n^T\otimes I_s & I_s \\
\end{array}
\right]\\
\nonumber&&=\left[\begin{array}{c|c} P & Q
\\\hline
R & S \\
\end{array}
\right]\\
\nonumber&&=-\frac{1}{2}L^*+\frac{1}{2}\delta^*{\delta^*}^T\otimes
H_n^{-1}
\end{eqnarray}
}
Now suppose $H_{n-1}=\sum_{i=1}^{n-1}W_i$ is not invertible. We
chose $\epsilon>0$ such that $W_1^*=W_1+\epsilon I$,
$H^*_{n-1}=\sum_{i=1}^{n-1}W_i+\epsilon I$, and
$H^*_{n}=\sum_{i=1}^{n}W_i+\epsilon I$ are simultaneously
invertible. Then using the new weights of the tree $T^*$ inverse
formula holds. Now by observing (\ref{lhs}) and (\ref{rhs}) we can
see that the term $H_{n-1}^{-1}$ is omitted in both the equation.
Hence by the continuity argument of matrices we can say that as
$\epsilon\rightarrow0$, $W_1^*\rightarrow W_1$,
$(W_1^*)^{-1}\rightarrow W_1^{-1}$, $(H^*_{n})^{-1}\rightarrow
H_{n}^{-1}$ and $H^*_{n}\rightarrow H_{n}$. Hence our inverse
formula is true for this case.
\end{proof}

We note that the formula for inverse of the distance matrix of a
weighted tree with the weights are invertible elements of a ring is given in \cite{zhou1}. One may derive the formula given from \cite[Theorem 10]{zhou1}. However, we present an independent proof here.   Next, let us illustrate the above theorem by an example.
\begin{figure}
\centering
\includegraphics[width=1.8in]{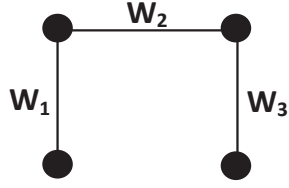}
\caption{Path of order 4}
\end{figure}
\begin{exam}\label{eaxm1}\emph{Consider a path of order 4 (Figure 3) , where the edge weights are}
$W_1=\left[
       \begin{array}{cc}
         2 & 0 \\
         0 & 1 \\
       \end{array}
     \right],
 W_2=\left[
       \begin{array}{cc}
         0 & 2 \\
         1 & 0 \\
       \end{array}
     \right],
 W_3=\left[
       \begin{array}{cc}
         1 & 0 \\
         0 & 2 \\
       \end{array}
     \right]
$ \emph{and then} $\sum_{i=1}^{3}W_i=\left[
       \begin{array}{cc}
         3 & 2 \\
         1 & 3 \\
       \end{array}
     \right]$. \emph{According to Theorem \ref{main}, the distance matrix of
     the tree is invertible. In this case} $\delta=[1,0,0,1]^T$ \emph{and
     distance matrix is given by} $$D=\left[
                                       \begin{array}{cccc}
                                         0 & W_1 & W_1+W_2 & W_1+W_2+W_3 \\
                                         W_1 & 0 & W_2 & W_2+W_3 \\
                                         W_1+W_2 & W_2 & 0 & W_3 \\
                                         W_1+W_2+W_3 & W_2+W_3 & W_3 & 0 \\
                                       \end{array}
                                     \right]=\left[
                                               \begin{array}{cccccccc}
                                                 0 & 0 & 2 & 0 & 2 & 2 & 3 & 2 \\
                                                 0 & 0 & 0 & 1 & 1 & 1 & 1 & 3 \\
                                                 2 & 0 & 0 & 0 & 0 & 2 & 1 & 2 \\
                                                 0 & 1 & 0 & 0 & 1 & 0 & 1 & 2 \\
                                                 2 & 2 & 0 & 2 & 0 & 0 & 1 & 0 \\
                                                 1 & 1 & 1 & 0 & 0 & 0 & 0 & 2 \\
                                                 3 & 2 & 1 & 2 & 1 & 0 & 0 & 0 \\
                                                 1 & 3 & 1 & 2 & 0 & 2 & 0 & 0 \\
                                               \end{array}
                                             \right]
     $$
\emph{Also the Laplacian matrix is given by}
$$L=\left[
                                       \begin{array}{cccc}
                                         W_1^{-1} & -W_1^{-1} & 0 & 0 \\
                                         -W_1^{-1} & W_1^{-1}+W_2^{-1} & -W_2^{-1} & 0 \\
                                         0 & -W_2^{-1} & W_2^{-1}+W_3^{-1} & -W_3^{-1} \\
                                         0 & 0 & -W_3^{-1} & W_3^{-1} \\
                                       \end{array}
                                     \right]=\left[
                                             \begin{array}{cccccccc}
                                                 .5 & 0 & -.5 & 0 & 0 & 0 & 0 & 0 \\
                                                 0 & 1 & 0 & -1 & 0 & 0 & 0 & 0 \\
                                                 -.5 & 0 & .5 & 1 & 0 & -1 & 0 & 0 \\
                                                 0 & -1 & .5 & 1 & -.5 & 0 & 0 & 0 \\
                                                 0 & 0 & 0 & -1 & 1 & 1 & -1 & 0 \\
                                                 0 & 0 & -.5 & 0 & .5 & .5 & 0 & -.5 \\
                                                 0 & 0 & 0 & 0 & -1 & 0 & 1 & 0 \\
                                                 0 & 0 & 0 & 0 & 0 & -.5 & 0 & .5 \\
                                               \end{array}
                                             \right]
     $$
\emph{Then one can verify that}
$D^{-1}=-\frac{1}{2}L+\frac{1}{2}\delta\delta^T\otimes(W_1+W_2+W_3)^{-1}$.
\end{exam}
As a consequence of the above theorem, we obtain one of the main results presented in \cite{balaji}.
\begin{coro}\emph{\cite[Theorem 3.7]{balaji}}
Let $T$ be a tree with $n$ vertices and positive definite matrix
weights. Suppose $D$ and $L$ are respectively the distance and
Laplacian matrices of $T$. Let $\delta_i$ denote the degree of $i$th
vertex of $T$ and
$\tau=(2-\delta_1,2-\delta_2,\cdots,2-\delta_n)^T$. If
$\Delta=\tau\otimes I_s$ and $R$ is the sum of all the weights of
$T$, then $$D^{-1}=-\frac{1}{2}L+\frac{1}{2}\Delta R^{-1}\Delta^T.$$
\end{coro}

Now, we present extensions of some of the results presented in \cite{bapatbook}. Among these results, the first four parts are true for trees with nonsingular matrix weights and the last
part holds true for trees with positive definite matrix weights.
\begin{theo}\label{extra}
Let $T$ be a tree with $n$ vertices, let $W_i$, $i=1,2,\dots,n-1$,
be the $s\times s$ edge weight matrices, and let $D$ be the distance
matrix of $T$. Let $D$ be invertible and let $L$ denote the
Laplacian matrix for weighting of $T$ in which each edge weight
matrix is replaced by its inverse. For $i=1,2,\dots,n$, set
$\delta_i=2-d_i$, where $d_i$ is the degree of the vertex $i$, and
$\delta^T=[\delta_1,\delta_2,\dots,\delta_n]$. Let $J$ denote the
matrix of appropriate size with all entries equals to $1$. Then the
following hold:
\begin{eqnarray*}(i)&&LD=\delta\textbf{\emph{1}}^T\otimes I_s-2I_n\otimes I_s~~~~~~~~~~~~~~~~~~~~~~~~~~~~~~~~~~~~~~~~~~~~~~~~~~~~~~~~~~~~~~~~~~~~~~~~~~~~~~~~~~~~~\\
(ii)&&DL=\textbf{\emph{1}}\delta^T\otimes I_s-2I_n\otimes I_s \\
(iii)&&LDL=-2L \\
(iv)&&(D^{-1}-L)^{-1}=\frac{1}{3}D+\frac{1}{3}J\otimes\sum_{i=1}^{n-1}W_i\\
(v)&&\mbox{If all the weight matrices are positive definite then $Q^TDQ=-2I_{(n-1)s\times(n-1)s}$, where $Q$ denotes the}\\
&&\mbox{incidence matrix of the weighted tree $T$ by replacing each
edge weight matrix by its inverse.}
\end{eqnarray*}
\end{theo}
\begin{proof}(i) As $D$ is invertible, so the matrices $W_i,~i=1,2,\dots,n-1$, and
$\sum_{i=1}^{n-1}W_i$ are invertible and
\begin{eqnarray}
\nonumber D^{-1}&=&-\frac{1}{2}L+\frac{1}{2}\delta\delta^T\otimes(\sum_{i=1}^{n-1}W_i)^{-1}\\
\nonumber\Rightarrow I_{ns}&=&-\frac{1}{2}LD+\frac{1}{2}\left[\delta\delta^T\otimes(\sum_{i=1}^{n-1}W_i)^{-1}\right]D\\
\label{ld}\Rightarrow
LD&=&\left[\delta\delta^T\otimes(\sum_{i=1}^{n-1}W_i)^{-1}\right]D-2I_{ns}
\end{eqnarray}
Again we have
\begin{eqnarray*}
(\textbf{1}^T\otimes I_s)D^{-1}&=&(\textbf{1}^T\otimes I_s)\left[-\frac{1}{2}L+\frac{1}{2}\delta\delta^T\otimes(\sum_{i=1}^{n-1}W_i)^{-1}\right]\\
&=&0+\frac{1}{2}\textbf{1}^T\delta\delta^T\otimes(\sum_{i=1}^{n-1}W_i)^{-1}\\
&=&\delta^T\otimes(\sum_{i=1}^{n-1}W_i)^{-1}\\
\Rightarrow (\textbf{1}^T\otimes
I_s)&=&\left[\delta^T\otimes(\sum_{i=1}^{n-1}W_i)^{-1}\right]D\\
\Rightarrow (\delta\otimes I_s)(\textbf{1}^T\otimes
I_s)&=&\left[\delta\delta^T\otimes(\sum_{i=1}^{n-1}W_i)^{-1}\right]D\\
\Rightarrow
\left[\delta\delta^T\otimes(\sum_{i=1}^{n-1}W_i)^{-1}\right]D&=&(\delta\textbf{1}^T\otimes
I_s)
\end{eqnarray*}
Hence from (\ref{ld}) we get
$$LD=\delta\textbf{1}^T\otimes I_s-2I_n\otimes I_s$$
(ii) Applying similar technique as in (i) we have the following,
$$DL=D\left[\delta\delta^T\otimes(\sum_{i=1}^{n-1}W_i)^{-1}\right]-2I_{ns},~D^{-1}(\textbf{1}\otimes I_s)=\delta\otimes(\sum_{i=1}^{n-1}W_i)^{-1}~\mbox{and}~ \textbf{1}\delta^T\otimes I_s=D\left[\delta\delta^T\otimes(\sum_{i=1}^{n-1}W_i)^{-1}\right]$$
Hence using these we get that
$$DL=\textbf{1}\delta^T\otimes I_s-2I_n\otimes I_s$$
(iii) \begin{eqnarray*} LDL&=&\left[\delta\textbf{1}^T\otimes
I_s-2I_n\otimes I_s\right]L\\
&=&\left[\delta\textbf{1}^T\otimes
I_s\right]L-2L\\
&=&\delta\left[(\textbf{1}^T\otimes
I_s)L\right]-2L\\
&=&0-2L\\
&=&-2L
\end{eqnarray*}
(iv) Let us assume $\sum_{i=1}^{n-1}W_i=H_{n-1}$. Then
\begin{eqnarray*} (D^{-1}-L)[\frac{1}{3}D+\frac{1}{3}J\otimes
H_{n-1}]&=&\frac{1}{3}D^{-1}D-\frac{1}{3}L
D+\frac{1}{3}D^{-1}(J\otimes H_{n-1})-\frac{1}{3}L (J\otimes H_{n-1})\\
&=&\frac{1}{3}I_{ns}-\frac{1}{3}[\delta\textbf{1}^T\otimes I_s-2I_n\otimes I_s]+\frac{1}{3}D^{-1}(\textbf{1}\textbf{1}^T\otimes H_{n-1})-\frac{1}{3}L (\textbf{1}\textbf{1}^T\otimes H_{n-1})\\
&=&I_{ns}-\frac{1}{3}\delta\textbf{1}^T\otimes I_s+\frac{1}{3}D^{-1}(\textbf{1}\otimes I_s)(\textbf{1}^T\otimes H_{n-1})-\frac{1}{3}L (\textbf{1}\otimes I_s)(\textbf{1}^T\otimes H_{n-1})\\
&=&I_{ns}-\frac{1}{3}\delta\textbf{1}^T\otimes
I_s+\frac{1}{3}(\delta\otimes H_{n-1}^{-1})(\textbf{1}^T\otimes
H_{n-1})-0\\
&=&I_{ns}-\frac{1}{3}\delta\textbf{1}^T\otimes
I_s+\frac{1}{3}\delta\textbf{1}^T\otimes
I_s\\
&=&I_{ns}
\end{eqnarray*}

Hence
$$(D^{-1}-L)^{-1}=\frac{1}{3}D+\frac{1}{3}J\otimes\sum_{i=1}^{n-1}W_i$$
(v) If all the weight matrices are positive definite, then the incidence
matrix $Q$ is well defined and Laplacian Matrix $L=QQ^T$. Since,
\begin{eqnarray*} LD&=&\delta\textbf{1}^T\otimes I_s-2I_n\otimes I_s, \\
\mbox{we have}~~ QQ^TDQ&=&(\delta\textbf{1}^T\otimes I_s)
Q-2(I_n\otimes
I_s) Q\\
&=&(\delta\otimes I_s)(\textbf{1}^T\otimes I_s) Q-2Q\\
&=&0-2Q\\
&=&-2Q
\end{eqnarray*}
Now as all the weight matrices are nonsingular, the incidence matrix
$Q$ has full column rank and thus it admits a left inverse. Hence
\begin{eqnarray*}Q^TDQ=-2I_{(n-1)s\times(n-1)s}
\end{eqnarray*}
\end{proof}

\begin{re}
\emph{Part(iii) of the above theorem gives an alternate simple proof
for a more general version of Lemma 3.2. of \cite{balaji}. The proof
presented in \cite{balaji} uses mathematical induction.}
\end{re}

A matrix $G$ of order $n\times m$ is said to be a generalized inverse (or a $g$-inverse) of $A$ if $AGA=A$.
Next, we shall consider the properties related to distance matrix, Laplacian matrix and g-inverse.
Let $e_{ij}$ be a $n\times 1$ vector with $i$ coordinate equal to 1,
$j$ coordinate equal to $-1$, and zeros elsewhere. Let $B$ be an
$ns\times ns$ matrix which is partitioned into the form
\begin{eqnarray}\label{partmatrix}B=\left(
                     \begin{array}{cccc}
                       B_{1,1} & B_{1,2} & \cdots & B_{1,n} \\
                       B_{2,1} & B_{2,2} & \cdots & B_{2,n} \\
                       \cdots & \cdots & \cdots & \cdots \\
                       B_{n,1} & B_{n,2} & \cdots & B_{n,n}
                     \end{array}
                   \right)
\end{eqnarray} where each $B_{i,j}$, for $i,j=1,2,\cdots,n$, is a submatrix of order
$s$. Then $(e_{ij}\otimes I_s)^TB(e_{ij}\otimes
I_s)=B_{i,i}+B_{j,j}-B_{i,j}-B_{j,i}$.

\begin{theo}\label{ginv}
Let $G$ be a weighted graph on $n$ vertices, where each
weight is a positive definite matrix. Let $D$ be the distance matrix
of $G$ and $L$ denote the Laplacian matrix of $G$ in which each edge
weight matrix is replaced by its inverse.

(i)If $H^1$ and $H^2$ are any two $g$-inverse of $L$, then
$(e_{ij}\otimes I_s)^TH^1(e_{ij}\otimes I_s)=(e_{ij}\otimes
I_s)^TH^2(e_{ij}\otimes I_s)$.

(ii) If $G$ is a tree on $n$ vertices and $H$ is a $g$-inverse of
$L$, then $H_{i,i}+H_{j,j}-H_{i,j}-H_{j,i}=D_{i,j}$.
\end{theo}
\begin{proof}(i) We have $Rank(L)=Rank(QQ^T)=Rank(Q)=ns-s$. Also the vectors $1\otimes
e_i$ for $i=1,2,\cdots,s$ form a  basis  for the
null space of $L$. Since $(\textbf{1}\otimes e_k)^T(e_{i,j}\otimes
I_s)=0$ for all $i,j,k\in\{1,2,\cdots,n\}$. Thus column space of
$e_{i,j}\otimes I_s$ is a subspace of column space of $L$. So there
exist a matrix $C$ such that $e_{i,j}\otimes I_s=LC$. Then
\begin{eqnarray*}(e_{ij}\otimes I_s)^T(H^1-H^2)(e_{ij}\otimes
I_s)=C^TL^T(H^1-H^2)LC=C^T(LH^1L-LH^2L)C=0
\end{eqnarray*}
Hence $(e_{ij}\otimes I_s)^TH^1(e_{ij}\otimes I_s)=(e_{ij}\otimes
I_s)^TH^2(e_{ij}\otimes I_s)$.

(ii) From Theorem \ref{extra}, we have  $LDL=-2L$. Thus
$-\frac{D}{2}$ is a $g$-inverse of $L$. Then by part (i)
$(e_{ij}\otimes I_s)^TH(e_{ij}\otimes I_s)=(e_{ij}\otimes
I_s)^T(-\frac{D}{2})(e_{ij}\otimes I_s)$.

So
$H_{i,i}+H_{j,j}-H_{i,j}-H_{j,i}=-\frac{1}{2}D_{i,i}-\frac{1}{2}D_{j,j}+\frac{1}{2}D_{i,j}+\frac{1}{2}D_{j,i}=D_{i,j}$
\end{proof}
Next let us recall a known result from \cite{bapat2}.
\begin{theo}\emph{\cite{bapat2}}
Let $T$ be a tree with n vertices, let $W_i$ be a positive definite
$s\times s$ edge weight matrix associated with the edge $e_i, i = 1,
2,\cdots, n-1$, and let $D$ be the distance matrix of $T$. Then $D$
has s positive and $(n - 1)s$ negative eigenvalues.
\end{theo}
Our next result gives some interlacing inequality for the
eigenvalues of $D$ and $L$ where the graph is a weighted tree on $n$
vertices. This result extends a result given in \cite{merr}.
\begin{theo}\label{interlace}
Let $T$ be a weighted tree on $n$ vertices, where each
weight is a positive definite matrix of order $s$. Let $D$ be the
distance matrix of $T$ and $L$ denote the Laplacian matrix of $T$ in
which each edge weight matrix is replaced by its inverse. Let
$\mu_1\geq\mu_2\geq\cdots\geq\mu_s>0>\mu_{s+1}\geq\cdots\geq\mu_{ns}$
be the eigenvalues of $D$ and
$\lambda_1\geq\lambda_2\geq\cdots\geq\lambda_{ns-s}>\lambda_{ns-s+1}=\cdots=\lambda_{ns}=0$
be the eigenvalues of $L$. Then
$$\mu_{s+i}\leq-\frac{2}{\lambda_i}\leq\mu_i~~\mbox{for}~~i=1,2,\cdots,ns-s.$$
\end{theo}
\begin{proof}Let $Q$ be the $ns\times (n-1)s$ vertex-edge incidence
matrix in which each edge weight matrix is replaced by its inverse.
As $Q$ has rank $(n-1)s$, columns of $Q$ are linearly
independent. By performing Gram-Schmidt process on these columns of
$Q$, we can say that there exist an $(n-1)s\times(n-1)s$ nonsingular
matrix $M$ such that the columns of $QM$ are orthonormal. As we have
$(\textbf{1}^T\otimes I_s)Q=0$,
$U=\left[QM,\frac{1}{\sqrt{n}}(\textbf{1}\otimes I_s)\right]$ is an
orthogonal matrix and so the eigenvalues of $D$ and $U^TDU$ are
equal. Now
\begin{eqnarray}U^TDU&=&\left[
          \begin{array}{cc}
            M^TQ^TDQM & \frac{1}{\sqrt{n}}M^TQ^TD(\textbf{1}\otimes I_s) \\
            \frac{1}{\sqrt{n}}(\textbf{1}^T\otimes I_s)DQM & \frac{1}{n}(\textbf{1}^T\otimes I_s)D(\textbf{1}\otimes I_s) \\
          \end{array}
        \right]\nonumber\\
        \label{inter}&=&\left[
          \begin{array}{cc}
            -2M^TM & \frac{1}{\sqrt{n}}M^TQ^TD(\textbf{1}\otimes I_s) \\
            \frac{1}{\sqrt{n}}(\textbf{1}^T\otimes I_s)DQM & \frac{1}{n}(\textbf{1}^T\otimes I_s)D(\textbf{1}\otimes I_s) \\
          \end{array}
        \right]~~\mbox{[By Theorem \ref{extra} Part (v)]}
\end{eqnarray}
Assume that $K=Q^TQ$, then $K$ is nonsingular. Also
$M^TKM=M^TQ^TQM=I$ implies $K^{-1}=MM^T$. As $M$ is nonsingular,
$K^{-1}$ and $M^TM$ have the same eigenvalues. Again the nonzero
eigenvalues of $L=QQ^T$ are the eigenvalues of $K=Q^TQ$. Thus
$-\frac{2}{\lambda_1}\geq-\frac{2}{\lambda_2}\geq\cdots\geq-\frac{2}{\lambda_{(n-1)s}}$
are the eigenvalues of $-2K^{-1}$ as well as $-2M^TM$. Hence by
interlacing theorem we have
$$\mu_{s+i}\leq-\frac{2}{\lambda_i}\leq\mu_i~~\mbox{for}~~i=1,2,\cdots,ns-s.$$
\end{proof}
Our next result, which is immediate from the previous theorem, gives an interlacing inequality for scaler weighted trees.
\begin{coro}Let $T$ be a weighted tree on $n$ vertices with each
weights are positive number. Let $D$ be the distance matrix of $T$
and $L$ be the Laplacian matrix of $T$ that arise by replacing each
weight by its reciprocal. Let $\mu_1>0>\mu_{2}\geq\cdots\geq\mu_{n}$
be the eigenvalues of $D$ and
$\lambda_1\geq\lambda_2\geq\cdots\geq\lambda_{n-1}>\lambda_{n}=0$ be
the eigenvalues of $L$. Then
$$0>-\frac{2}{\lambda_1}\geq\mu_2\geq-\frac{2}{\lambda_2}\geq\mu_3\geq\cdots\geq-\frac{2}{\lambda_{n-1}}\geq\mu_n.$$
\end{coro}
\textbf{Acknowledgement:} M. Rajesh Kannan would like to thank
Department of Science and Technology for the financial
support.

\bibliographystyle{amsplain}
\bibliography{inversedistbib}

%
%
%
%
%
%
%
%
%
%
%
%
%
%

\end{document}